\documentclass[reqno]{amsart}
\usepackage{amsfonts}
\usepackage{amsthm}
\usepackage{amstext}
\usepackage{amsmath}
\usepackage{amscd}
\usepackage{amssymb}
\usepackage[mathscr]{eucal}
\usepackage{url}
\usepackage{graphicx}
\usepackage{verbatim}

\newtheorem{theorem}{Theorem}
\newtheorem{corollary}[theorem]{Corollary}
\newtheorem{proposition}[theorem]{Proposition}
\newtheorem{lemma}[theorem]{Lemma}
\newtheorem*{theoremA}{Theorem~A}
\newtheorem*{theoremB}{Theorem~B}
\newtheorem*{theoremC}{Theorem~C}
\newtheorem*{theoremD}{Theorem~D}

\theoremstyle{definition}

\newtheorem{example}[theorem]{Example}
\newtheorem{remark}[theorem]{Remark}

\numberwithin{equation}{section}

\newcommand\mL{L\kern-0.08cm\char39}

\newcommand{\Id}{{\rm Id}}

\newcommand{\Orb}{{\rm Orb}}

\newcommand{\Fix}{{\rm Fix}}
\newcommand{\Per}{{\rm Per}}
\newcommand{\Rec}{{\rm Rec}}
\newcommand{\Min}{{\rm Min}}

\newcommand{\conv}{{\rm conv}}

\newcommand{\C}{{\mathcal C}}
\newcommand{\Homeo}{{\mathcal H}}

\newcommand{\tor}{\mathop{\rm tor}}

\begin{document}

\begin{large}

\title[Minimality for actions of abelian semigroups]{Minimality for actions of abelian semigroups\\ on compact spaces with a free interval}

\author{Mat\'u\v s Dirb\'ak}
\author{Roman Hric}
\author{Peter Mali\v ck\'y}
\author{\mL ubom\'\i r Snoha}
\author{Vladim\'\i r \v Spitalsk\'y}
\address{Department of Mathematics, Faculty of Natural Sciences, Matej Bel University,
Tajovsk\'{e}ho 40, 974 01 Bansk\'{a} Bystrica, Slovakia}

\email{[Matus.Dirbak, Roman.Hric, Peter.Malicky, Lubomir.Snoha, Vladimir.Spitalsky]@umb.sk}

\thanks{This work was supported by the Slovak Research and Development Agency under the contract No.~APVV-15-0439 and by VEGA grant 1/0786/15.}

\subjclass[2010]{Primary 37B05; Secondary 54H20}

\keywords{Abelian semigroup, minimal action, minimal set, free interval, free arc}

\dedicatory{Dedicated to Sergii Kolyada on the occasion of his 60th birthday}

\begin{abstract}
	We study minimality for continuous actions of abelian semigroups
	on compact Hausdorff spaces with a free interval.
	First, we give a necessary and sufficient condition for such a space to admit a minimal action of a given abelian semigroup.	
	Further, for actions of abelian semigroups we provide a trichotomy for the topological structure of minimal sets intersecting a free interval.
\end{abstract}

\maketitle


\section{Introduction}\label{S:intro}

Various aspects of the dynamics of group actions on the circle $\mathbb S^1$ by homeomorphisms are well understood. By the Ghys-Margulis alternative \cite{Mar},
every (effective) action of a group $G$ on the circle either has an
invariant probability measure or contains a free subgroup with two generators. Moreover, by Malyutin's theorem \cite{Mal}, if the action of $G$ is minimal then either $G$ is conjugate to a group of rotations or it is a finite cover of a proximal action (see also Glasner's paper \cite{Gla} for a shortened proof). As far as minimal sets of group actions on $\mathbb S^1$ are concerned, the following trichotomy
holds (see \cite[Proposition~5.6]{Ghy}, \cite[Theorem~3.7]{Bek04}, \cite[Theorem~3.3]{Gla}):
\begin{enumerate}
\item every minimal set is finite;
\item the whole circle $\mathbb S^1$ is a (unique) minimal set;
\item the action has a unique minimal set, which is a Cantor set.
\end{enumerate}

For actions of general semigroups $S$ on the circle by continuous maps, much less is known about minimality and minimal sets, an obvious exception being the completely understood case $S=\mathbb N$. Recall that every minimal map on $\mathbb S^1$ is conjugate to an irrational rotation. Moreover, the following full topological characterization of minimal sets of continuous circle maps holds: a subset of the circle is a minimal set for some circle map if and only if it is either finite or a Cantor set or the whole circle. Contrary to the above trichotomy for group actions on $\mathbb S^1$, a noninvertible circle map may have both finite and Cantor minimal sets.

The minimality for continuous maps is well understood even on spaces with a free interval.
Recall that a \emph{free interval} $J$ in a space $X$ is an open subset of $X$ homeomorphic to the open interval $(0,1)$. By \cite[Theorem~A]{DSS}, every minimal continuous map on a metric continuum $X$ with a free interval is conjugate to an irrational circle rotation. Further,
if $X$ is a compact metric space with a free interval $J$ and $f\colon X\to X$ is a continuous map
then, by \cite[Theorem~B]{DSS}, the following trichotomy holds for every minimal set $M$ of $f$ intersecting $J$:
\begin{enumerate}
	\item $M$ is finite;
	\item $M$ is a disjoint union of finitely many circles;
	\item $M$ is a nowhere dense cantoroid,
\end{enumerate}
where by a \emph{cantoroid} we mean a compact metric space with a dense set of degenerate components and without isolated points.

In the present paper we study minimality for actions, by continuous maps, of abelian semigroups on compact spaces with a free interval. First, we show that in the class of such spaces only finite disjoint unions of circles admit minimal actions of abelian semigroups.

\newcommand{\thmA}{
	Let $X$ be a compact Hausdorff space with a free interval and let $\Phi$ be a minimal action of an abelian semigroup $S$ on $X$. Then $X$ is a disjoint union of finitely many circles and all the acting maps of $\Phi$ are homeomorphisms of $X$.
}
\begin{theoremA}\label{T:cpct-free-disc}
	\thmA
\end{theoremA}

If $X$ is connected then we have the following strengthening of Theorem~A.

\newcommand{\thmB}{
	Let $X$ be a compact connected Hausdorff space with a free interval and let $\Phi$ be a minimal action of an abelian semigroup $S$ on $X$. Then $\Phi$ is conjugate to an action of $S$ by rotations on the circle $\mathbb S^1$.
}

\begin{theoremB}\label{T:cpct-free}
	\thmB
\end{theoremB}

Having Theorems~A and ~B at our disposal, we may ask whether a given abelian semigroup $S$ acts in a minimal way on a given disjoint union of circles $X$. Via Proposition~\ref{P:ex.min.gp}, this can be reduced to an algebraic characterization of the abelian subgroups of the group $\Homeo(X)$ of all homeomorphisms of $X$, whose natural actions on $X$ are minimal. The latter problem is addressed in the following theorem.

\newcommand{\thmC}{
	Let $\ell$ be a positive integer and $X$ be a disjoint union of $\ell$ circles. Given an abelian group $G$, the following conditions are equivalent:
	\begin{enumerate}
		\item $G$ is isomorphic to a subgroup of $\Homeo(X)$, whose natural action on $X$ is minimal;
		\item there is a short exact sequence of abelian groups
		\begin{equation*}
		0\longrightarrow H\longrightarrow G\longrightarrow K\longrightarrow0
		\end{equation*}
		such that $H$ is (isomorphic to) a dense subgroup of $\mathbb S^1$ and $K$ has cardinality $\ell$.
	\end{enumerate}
}

\begin{theoremC}\label{T:exist.min.act}
	\thmC
\end{theoremC}

Finally, we study minimal sets of abelian semigroup actions which intersect a free interval and prove the following trichotomy in analogy to that from \cite[Theorem~B]{DSS}. Let us mention that for a subset $B$ of a free interval $J$, we denote by $\conv(B)$ the convex hull of $B$ in $J$ (see also Section~\ref{S:cpct-free}).

\newcommand{\thmD}{
	Let $X$ be a compact Hausdorff space with a free interval $J$ and let $\Phi$ be an action of an abelian semigroup $S$ on $X$. Assume that $M\subseteq X$ is a minimal set for $\Phi$, which intersects $J$. Then $M$ is contained in a closed metrizable locally connected subspace of $X$ and exactly one of the following conditions holds:
	\begin{enumerate}
		\item $M$ is finite;
		\item $M$ is a disjoint union of finitely many circles;
		\item $M$ is nowhere dense in $X$, $A=\overline{\conv(M\cap J)}$ is an arc or a circle and $M\cap A$ is a Cantor set.
	\end{enumerate}
}

\begin{theoremD}\label{T:min.set.free.int}
	\thmD
\end{theoremD}

If all the acting maps of $\Phi$ are homeomorphisms, this trichotomy can be slightly strengthened, see Corollary~\ref{C:min.set.free.int}.

These results do not generalize to actions of arbitrary semigroups. For instance, by Duminy's theorem, the free semigroup on two generators has a minimal action on the interval $[0,1]$, see e.g.~\cite[Section~3.3]{Nav} and \cite{Sh}. It would be interesting to know to what extent our results generalize to actions of non-abelian semigroups.

\section{Minimal abelian semigroup actions on topological spaces}\label{S:Sactions}

\subsection{Definitions}\label{SS:Defs}
The set of positive integers is denoted by $\mathbb N$. By a space we mean a topological space. For the composition of two selfmaps $f,g$ of a space $X$, we write both $f\circ g$ and $fg$.

Let $S$ be a (nonempty) abelian semigroup (equipped with the discrete topology) and $X$ be a space.
By an \emph{action} of $S$ on $X$ we mean a continuous map
$\Phi:S\times X\to X$ such that $\Phi(s,\Phi(t,x)) = \Phi(s+t,x)$ for all $s,t\in S$
and $x\in X$. The acting maps $\Phi(s,\cdot)$ of $\Phi$ will be denoted by $\varphi_s:X\to X$.
If the semigroup $S$ is a group with zero element $0$ and we explicitly speak about a group action, then it is additionally assumed that $\Phi(0,x)=x$ for every $x\in X$;
in such a case, all the acting maps $\varphi_s$ of the group action $\Phi$ are homeomorphisms of $X$.

An action $\Phi$ of $S$ on $X$ is \emph{minimal} if for every $x\in X$, the \emph{$\Phi$-orbit}
$\Orb_\Phi(x):=\{\varphi_s(x): s\in S\}$ of $x$ is dense in $X$.
An equivalent definition
is that $X$ is the unique minimal set of $\Phi$. Recall that a subset $M$ of $X$ is a \emph{minimal set} of an action $\Phi$ if it is nonempty, closed, $\Phi$-invariant
(that is, $\varphi_s(M)\subseteq M$ for every $s\in S$), and there is no proper subset with these three properties.

We say that an action $\Phi$ of $S$ on $X$ is \emph{free} or that $S$ \emph{acts freely} on $X$ if,
for every $s\in S$, either the acting map $\varphi_s$ is the identity or it has no fixed point.
If $S$ is a group then the action of $S$ on $X$ is called \emph{effective}  if $\varphi_s=\Id_X$ implies $s=0$.

For a space $X$, denote by $\C(X)$ the semigroup of all continuous maps $f\colon X\to X$,
and by $\Homeo(X)$ the group of all homeomorphisms of $X$.
In the case when $S$ is a subsemigroup of $\C(X)$,
we shall often consider the \emph{natural action} $\Phi$ of $S$ on $X$
with $\varphi_s=s$ for every $s\in S$.

Two actions $\Phi,\Psi$ of $S$ on $X,Y$ with acting maps $\varphi_s,\psi_s$, respectively,
are said to be \emph{(topologically) conjugate}
if there is a homeomorphism $h:X\to Y$ such that
$h\circ \varphi_s = \psi_s\circ h$ for every $s\in S$.
Clearly, if $\Phi,\Psi$ are conjugate and one of them is minimal then so is the other one.

\subsection{Basic facts on minimal semigroup actions}\label{SS:Sfacts}

If $(X,f)$ is a dynamical system given by a space $X$ and a continuous map $f\colon X\to X$, and a (not necessarily closed) subset $A$ of $X$ is \emph{$f$-invariant}, i.e.~$f(A)\subseteq A$,
then we also use a less precise notation $(A,f)$ for the subsystem $(A, f|_A)$ of $(X,f)$.
By $\Fix (f)$, $\Per (f)$, $\Min (f)$ and $\Rec (f)$  we denote the set of fixed, periodic, minimal and recurrent points of $f$, respectively (recall that a \emph{minimal point} is an element of a minimal set).

\begin{lemma}\label{L:facts X}
	Let $\Phi$ be a minimal action of an abelian semigroup $S$ on a topological space $X$
	and $f,g$ be two of the acting maps of $\Phi$. Then the following statements hold:
	\begin{enumerate}
		\item\label{IT:3}  if $\emptyset \neq A\subseteq X$ is $\Phi$-invariant then it is dense in $X$;
		
		\item\label{IT:4} $(g(X), f)$ is both a subsystem of $(X,f)$ and a factor of $(X,f)$;
		
		\item\label{IT:5} if $A\subseteq X$ is $f$-invariant then
		 $(g(A),f)$ is both a subsystem of $(X,f)$ and a factor of $(A,f)$, in particular, if $A$ is a minimal set of $f$ then so is $g(A)$;
		
		\item\label{IT:9} the set $g(X)$ is $\Phi$-invariant and dense in $X$;
		
		\item\label{IT:7}
			the sets $\Fix (f)$, $\Per (f)$, $\Min (f)$
			and $\Rec (f)$ are $\Phi$-invariant,
			and if any of them is nonempty then it is dense in $X$.
			
	\end{enumerate}
\end{lemma}

\begin{proof}	
	Statement (\ref{IT:3}) follows from definition of minimality and (\ref{IT:4}), (\ref{IT:5}) follow from the identity $f\circ g=g\circ f$. Further, (\ref{IT:9}) follows from (\ref{IT:4}) and (\ref{IT:3}). 	 	
	We check (\ref{IT:7}). By virtue of (\ref{IT:5}), if a point $x$ belongs to one of these sets, then so does its image under every acting map of $\Phi$. Consequently, all these sets are $\Phi$-invariant and so the density part of statement (\ref{IT:7}) follows from (\ref{IT:3}).
\end{proof}

\begin{lemma}\label{L:facts H}
	Let $\Phi$ be a minimal action of an abelian semigroup $S$
	on a Hausdorff space~$X$. Then $\Phi$ is free.
\end{lemma}

\begin{proof}
  Let $s\in S$ and assume that $\varphi_s$ has a fixed point. Then, by Lemma~\ref{L:facts X}(\ref{IT:7}), $\Fix(\varphi_s)$ is dense in $X$.
  Since $X$ is Hausdorff, the continuity of $\varphi_s$ gives $\Fix(\varphi_s)=X$ and hence $\varphi_s=\Id_X$ .
\end{proof}

\begin{corollary}\label{C:fixed-point-property}
	If a non-degenerate Hausdorff space $X$ has the fixed point property,
	then it does not admit a minimal action of any abelian semigroup $S$.
\end{corollary}

We shall have an occasion to use the following well known lemma. We include a proof for completeness.

\begin{lemma}\label{L:dens.back.orb}
Let $\Phi$ be a minimal action of an abelian semigroup $S$ on a compact Hausdorff space $X$. Then for every nonempty open set $V\subseteq X$ there exist finite sets $E,F\subseteq S$ with $X=\bigcup_{t\in E}\varphi_t^{-1}(V)$ and $X=\bigcup_{r\in F}\varphi_r(V)$.
\end{lemma}
\begin{proof}
By minimality of $\Phi$, $X=\bigcup_{t\in S}\varphi_t^{-1}(V)$ and so the existence of the required set $E$ follows by compactness of $X$. Further, by virtue of Lemma~\ref{L:facts X}(\ref{IT:9}) and our assumptions on $X$, all the acting maps of $\Phi$ are surjective. Now we may assume that $V$ is a proper subset of $X$, otherwise the rest of the proof is trivial. Thus, the set $E$ has at least two elements.
Set $p=\sum_{s\in E}s$, $r_t=\sum_{s\in E\setminus t}s$ for $t\in E$ and $F=\{r_t : t\in E\}$. Since $p=r_t+t$ for every $t\in E$, we obtain
\begin{equation*}
X=\varphi_p(X)=\bigcup_{t\in E}\varphi_p\left(\varphi_t^{-1}(V)\right)=\bigcup_{t\in E}\varphi_{r_t}(V)=\bigcup_{r\in F}\varphi_r(V),
\end{equation*}
for $\varphi_t(\varphi_t^{-1}(V))=V$ by surjectivity of $\varphi_t$.
\end{proof}

\begin{lemma}\label{L:min.set.free.int}
	Let $\Phi$ be a minimal action of an abelian semigroup $S$ on a compact Hausdorff space $X$. Then the following statements hold:
	\begin{enumerate}
		\item if $\varphi_s(V)$ is a singleton for a nonempty open set $V\subseteq X$ and some $s\in S$ then the space $X$ is finite;
		\item if the space $X$ has an isolated point then it is finite.
	\end{enumerate}
\end{lemma}
\begin{proof}
	Fix $V$ and $s$ as in (1) and write $\varphi_s(V)=\{x\}$. By Lemma~\ref{L:facts X}(\ref{IT:9}), the open set $\varphi_s^{-1}(V)$ is nonempty. Since $\Phi$ is minimal, we have $\varphi_t(x)\in\varphi_s^{-1}(V)$ for some $t\in S$. Then $x=\varphi_s\varphi_s\varphi_t(x)=\varphi_{2s+t}(x)$ and so $\varphi_{2s+t}=\varphi_{s+t}\varphi_s$ is the identity on $X$ by virtue of
	 Lemma~\ref{L:facts H}. This means that $\varphi_s$ is an injection, whence it follows that $V$ is a singleton. Now, by Lemma~\ref{L:dens.back.orb}, there is a finite set $F\subseteq S$ with $X=\bigcup_{r\in F}\varphi_r(V)$. Thus, the space $X$ is finite, which verifies statement (1). Statement (2) follows immediately from statement (1).
\end{proof}

Let $\Phi$ be a minimal action of an abelian semigroup $S$ on a space $X$. Assume that the acting maps $\varphi_s$ are homeomorphisms of $X$, thus forming
an (abelian) \emph{subsemigroup} $S_\Phi=\{\varphi_s: s\in S\}$ of the group $\Homeo(X)$ of all homeomorphisms of $X$. Define
\begin{equation}\label{Eq:Gphi}
G_\Phi = \left\{
\varphi_{s} \circ   \varphi_{t}^{-1}:
s,t\in S
\right\}.
\end{equation}
Then
\begin{itemize}
	\item $G_\Phi$ is an \emph{abelian subgroup} of $\Homeo(X)$,
	\item $G_\Phi\supseteq S_\Phi$; indeed, given $s\in S$, we have
	$\varphi_s=\varphi_{s+s}\circ \varphi_s^{-1}\in G_\Phi$.
\end{itemize}
Thus, $G_\Phi$ is the subgroup of $\Homeo(X)$
generated by $S_\Phi$.


\section{Minimal abelian semigroup actions on compact spaces with a free interval}\label{S:cpct-free}

While in the statements of our main theorems we use the notion of a free interval,
now it will be convenient to use also the notion of a free arc.
Recall that a \emph{free arc} $A$ in a space $X$ is  a subset of $X$ homeomorphic to the compact unit interval $[0,1]$, which becomes an open set in $X$ after removing its two end points.
We shall often use an identification of a free arc (free interval) with a genuine compact (open) interval in the real line, together with the natural order.
In particular, the notion of the convex hull $\conv(M)$ is naturally defined for a subset $M$ of a free interval or a free arc.

Assume that a map $f$ sends a closed arc $[a,b]$ onto a closed arc $[c,d]$. If $f(a)=c$ and $f(b)=d$ then we write $f\colon[a,b]\nearrow[c,d]$. Similarly, if $f(a)=d$ and $f(b)=c$ then we write $f\colon[a,b]\searrow[c,d]$.

\begin{lemma}\label{L:eaist.I'}
	Let $X$ be a Hausdorff space with a free arc $A=[0,1]$ and let $f\colon X\to X$ be a continuous map without fixed points in $A$. Assume that $a,b\in [0,1]$ are such that $f(a)\in(0,1)$  and $f(b)\notin[0,1]$. Then the following statements hold:
	\begin{enumerate}
		\item if $a<b$ and $f(a)>a$ then there is an arc $I\subseteq [a,b]$ with $f\colon I\nearrow[f(a),1]$;
		\item if $a<b$ and $f(a)<a$ then there is an arc $I\subseteq [a,b]$ with $f\colon I\searrow[0,f(a)]$;
		\item if $b<a$ and $f(a)>a$ then there is an arc $I\subseteq [b,a]$ with $f\colon I\searrow[f(a),1]$;
		\item if $b<a$ and $f(a)<a$ then there is an arc $I\subseteq [b,a]$ with $f\colon I\nearrow[0,f(a)]$.
	\end{enumerate}
\end{lemma}
\begin{proof}
	Before turning to the proof, notice that
	the set $(0,1)$ is open in $X$ by definition of a free arc and the set $[0,1]$ is (compact hence) closed in $X$ by the Hausdorff property of $X$.
	
	Now, without loss of generality, we may assume that $a<b$. (The other case $b<a$ then follows by an obvious symmetry argument.) The set $f^{-1}((0,1))$ is open in $X$ and so $U=[a,b]\cap f^{-1}((0,1))$ is open in $[a,b]$. Let $C$ be the component of $U$ containing $a\in U$. Since $b\notin U$, the set $C$ is of the form $C=[a,z)$ for some $a<z<b$. Further, as observed above, the sets $(0,1)$ and $X\setminus[0,1]$ are open in $X$ and so $f(z)\in\{0,1\}$. Finally, using that $A\cap\Fix(f)=\emptyset$, it is sufficient to set $u=\max\{t\in[a,z] : f(t)=f(a)\}$ and $I=[u,z]$.
\end{proof}

\begin{lemma}\label{DSS-A.1}
	Let $X$ be a Hausdorff space with a free arc $A=[0,1]$ and $h:X\to X$ be a continuous map. Let $0 \leq \alpha < \beta \leq \gamma < \delta \leq 1$ be such that $h([\alpha,\beta]\cup [\gamma,\delta]) \subseteq [0,1]$, $h(\alpha) \geq \delta$, $h(\gamma) \leq \alpha$ and $h(\delta) \in h([\alpha,\beta])$.
	Then $\Fix(h) \cap [\gamma,\delta] \neq \emptyset$ or $\Per(h) \cap [\alpha,\beta] \neq \emptyset$.
\end{lemma}
\begin{proof}
	Suppose that $\Fix(h) \cap ([\alpha,\beta] \cup [\gamma,\delta]) = \emptyset$.
	Then $h(x) > x$ on $[\alpha,\beta]$ and $h(x) < x$ on $[\gamma,\delta]$.
	By replacing $\beta$ with an appropriate point from $(\alpha,\beta]$, if necessary, we may assume that $\beta<h(\beta)=h(\delta)<\delta$ and $h(x)\geq h(\delta)$ for every
	$x \in [\alpha,\beta]$. Set $\tilde{\alpha} = \max \{t \in [\alpha,\beta] : h(t) = \delta \}$ and $\tilde{\gamma} = \max \{t \in [\gamma,\delta] : h(t) = \tilde\alpha \}$.
	By applying \cite[Lemma~A.1]{DSS} to $f=h|_{[\tilde\alpha,\beta]} \colon [\tilde\alpha,\beta] \to
	[\tilde\alpha,\delta]$ and $g=h|_{[\tilde\gamma,\delta]} \colon [\tilde\gamma,\delta] \to
	[\tilde\alpha,\delta]$, we obtain $\Per(h) \cap [\tilde\alpha,\beta] \neq \emptyset$.
\end{proof}

\begin{lemma}\label{L:recper2}
	Let $X$ be a Hausdorff space with a free arc $A=[0,1]$ and $f:X\to X$ be a continuous map without periodic points in $A$. Let $a,b\in A$ be such that $0<a < f(a)\leq f(b) < b<1$. Then for every $x\in A$:
	\begin{enumerate}
		\item\label{L:recper2:1} if $x<a$ then $x<f(x)\leq b$;
		\item if $x>b$ then $x>f(x)\geq a$.
	\end{enumerate}
\end{lemma}
\begin{proof}
	By a simple symmetry argument it is sufficient to verify (\ref{L:recper2:1}). On the contrary, assume that
	$f(x) \notin (x,b]$ for some $0 \leq x <a$. Then in fact $f(x)\not\in[x,b]$,
	for $f$ has no fixed points in $A$.
	Since $f(a) \in (x,b)$, Lemma~\ref{L:eaist.I'}(3) yields the existence of
    an arc $I' \subseteq [x,a]$ such that $f \colon I' \searrow [f(a),b]$.
	Further, since $f$ has no fixed point in $[a,b]$, there is $y \in (a,b)$ such that $f(y) \notin [0,1]$.
	As $f(y) \notin [x,b]$ and $f(b) \in (x,b)$, by Lemma~\ref{L:eaist.I'}(4) there is an arc $I'' \subseteq [y,b]$ such that $f \colon I'' \nearrow [x,f(b)]$.
	Applying now Lemma~\ref{DSS-A.1} to $[\alpha,\beta]=I'$ and $[\gamma,\delta]=I''$ yields a periodic point
	of $f$ in $[x,b]$, a contradiction.
\end{proof}

\begin{lemma}\label{L:arc-orb}
	Let $X$ be a Hausdorff space with a free arc $A=[0,1]$ and $f:X\to X$ be a continuous map without periodic points in $A$. Assume that $0<f(0)=c=f(1)<1$. Then $\Orb_f(c)\subseteq(0,1)$.
\end{lemma}
\begin{proof}
	By contradiction. Let $n$ be the smallest positive integer with $f^n(c) \notin (0,1)$. Then
	$f^n(c) \notin [0,1]$, for otherwise 0 or 1 would be a periodic point of $f$. Put $d=f^{n-1}(c) = f^n(0) = f^n(1) \in (0,1)$.
	Without loss of generality, we may assume that  $d\le c$. By applying Lemma~\ref{L:eaist.I'}(4) to $f^n$,
	$a=1$ and $b=c$, we find an arc $I' \subseteq [c,1]$ such that $f^n \colon I' \nearrow [0,d]$.
	Applying Lemma~\ref{L:eaist.I'}(1) to $f$, $a=0$ and $b=d$ yields an arc $I'' \subseteq [0,d]$ with
	$f \colon I'' \nearrow [c,1]$. Consequently, $I'$ contains an arc $K$ such that
	$f^{n+1}(K) = [c,1] \supseteq K$. Thus $f^{n+1}$ has a fixed point in $K$, a contradiction.
\end{proof}

\begin{lemma}\label{L:homeo-J}
	Let $X$ be a Hausdorff space with a free interval $J$ and let $\Phi$ be a minimal action of an abelian semigroup $S$ on $X$. Then all the acting maps $\varphi_s$ of $\Phi$ are injective on $J$.
\end{lemma}
\begin{proof}
	We shall proceed by contradiction. So assume that $s\in S$ and $a,b\in J$ are such that $a<b$ and $\varphi_s(a)=\varphi_s(b)$. By minimality of $\Phi$, there is $r\in S$ with $c:=\varphi_r\varphi_s(a)=\varphi_r\varphi_s(b)\in(a,b)$. Set $t=r+s\in S$. Notice that $\varphi_t$ has no periodic points. (Indeed, if $(\varphi_t)^k(y)=y$ for some $k\geq1$ and $y\in X$ then $(\varphi_t)^k$ would be the identity by Lemma~\ref{L:facts H}, hence $\varphi_t$ would be injective.) Thus, we may apply Lemma~\ref{L:arc-orb} to $f=\varphi_t$ and $A=[a,b]$ to obtain $\Orb_{\varphi_t}(c)\subseteq(a,b)$.
	
	Since the closure of $\Orb_{\varphi_t}(c)$ is compact, the set $\Min(\varphi_t)$ is nonempty and, by Lemma~\ref{L:facts X}(\ref{IT:7}), it is dense in $X$. Consequently, $\varphi_t$ has a minimal set $M$, which intersects $J$ on the left of $a$. As $M$ is not a periodic orbit for $\varphi_t$, it has no isolated points. Therefore, there exist $z,x\in M\cap J$ with $z<x<a$. Now, given $n\in\mathbb N$, we have $a<(\varphi_t)^n(a)=(\varphi_t)^n(b)<b$. Thus, by Lemma~\ref{L:recper2}(1) applied to $f=(\varphi_t)^n$, we obtain $x<(\varphi_t)^n(x)\leq b$. Since the latter holds for every $n\in\mathbb N$ and since $z<x$, we infer that $z$ is not contained in the orbit closure of $x$ under the action of $\varphi_t$. This contradicts the fact that $M$ is a minimal set of $\varphi_t$, which finishes the proof.
\end{proof}

\begin{lemma}\label{L:fin.many.arcs}
	Let $X$ be a Hausdorff space, which can be expressed as a union of finitely many arcs. Then the union of the free intervals of $X$ is dense in $X$.
\end{lemma}
\begin{proof}
	Obviously, it suffices to prove the following claim:
	\begin{itemize}
		\item Let $X=Y\cup Z$ be a compact Hausdorff space, where $Y$ and $Z$ are closed subspaces of $X$, each of them having a dense union of free intervals in the relative topology. Then $X$ also has a dense union of free intervals.
	\end{itemize}
	To prove this claim, let $U$ be a nonempty open subset of $X$. We want to find a free interval $J$ of $X$ contained in $U$. We distinguish three cases.
	
	First assume that $U\cap(Y\setminus Z)\neq \emptyset$. The set $U\setminus Z$ is nonempty and open in $X$ and is contained in $Y$. By the assumption on $Y$, $U\setminus Z$ contains a free interval $J$ of $Y$. Clearly, $J$ is open in $U\setminus Z$ and hence also in $X$. Thus $J$ is a free interval of $X$ contained in $U$. The case $U\cap(Z\setminus Y)\neq \emptyset$ is handled analogously.
	
	It remains to consider the case $U\subseteq Y\cap Z$. Since $U\subseteq Y$, there is a free interval $J$ of $Y$ in $U$. The set $J$ is open in $U$, hence in $X$ and so it is a free interval of $X$ contained in $U$.
\end{proof}

\begin{theoremA}
	\thmA
\end{theoremA}
\begin{proof}
First observe that $X$ is a union of finitely many arcs by Lemmas~\ref{L:dens.back.orb}	 and~\ref{L:homeo-J} and so, by Lemma~\ref{L:fin.many.arcs}, the union of the free intervals of $X$ is dense in $X$.
	
	We begin the proof by showing that the acting maps $\varphi_s$ of $\Phi$ are homeomorphisms of $X$. Since all $\varphi_s$ are surjective by Lemma~\ref{L:facts X}(\ref{IT:9}) and by compactness of the Hausdorff space $X$, it is sufficient to show that they are injective. So assume, on the contrary, that $\varphi_s(a)=\varphi_s(b)$ for some $s\in S$ and some distinct points $a,b\in X$. Use Lemma~\ref{L:dens.back.orb} to find $r\in S$, a free interval $J\subseteq X$ and $x\in J$ with $\varphi_r(x)=a$. Also, fix $y\in X$ with $\varphi_r(y)=b$ and notice that $y\neq x$. By minimality of $\Phi$, there are $t\in S$ and a free interval $J'\subseteq X$ with $\varphi_t\varphi_s(a)=\varphi_t\varphi_s(b)\in J'$. Set $p=r+s+t$.
	
	Fix a neighbourhood $W$ of $y$ in $X$ with $\varphi_p(W)\subseteq J'$
	and choose a free interval $I\subseteq W$ of $X$. Since $\varphi_p$ has no fixed points
	by Lemma~\ref{L:facts H}, the points $x$, $y$ and $\varphi_p(x)=\varphi_p(y)$ are mutually distinct and so we may assume that $J,I,J'$ are mutually disjoint. Moreover, due to Lemma~\ref{L:homeo-J}, we may also suppose that $\varphi_p\colon J\to J'$ is a homeomorphism and $\varphi_p\colon I\to J'$ is an open map (in fact, a homeomorphism onto the open subset $\varphi_p(I)$ of $J'$).

	By Lemma~\ref{L:facts X}(\ref{IT:7}) and compactness of $X$, $I$ intersects a minimal set $M$ for $\varphi_p$. We claim that
	\begin{enumerate}
		\item[($*$)] $K\cap\varphi_p^{-1}(M)\subseteq M$ for every free interval $K$ of $X$.
	\end{enumerate}
	Indeed, since the map $\varphi_p$ has no periodic points and it has dense minimal (hence recurrent) points, \cite[Theorem~20]{DSS} yields $K\cap\varphi_p^{-1}(M)\subseteq\Rec(\varphi_p)\cap\varphi_p^{-1}(M)\subseteq M$, where the last inclusion follows from minimality of $M$ for $\varphi_p$.
	
	Now set $U=\varphi_p(I\cap M)$ and $V=J\cap\varphi_p^{-1}(U)$. Clearly, both $U$ and $V$ are nonempty; we want to verify that they are open subsets of $M$. First, by applying ($*$) to $K=I$, we get $I\cap\varphi_p^{-1}(M)\subseteq M$, whence $U=\varphi_p(I\cap M)=\varphi_p(I)\cap M$ is open in $M$. Further, by applying ($*$) to $J$, we obtain $V=J\cap\varphi_p^{-1}(U) \subseteq M$ and hence $V = J\cap (\varphi_p|_M)^{-1}(U)$ is also an open subset of $M$.
	
	We want to arrive at a contradiction by showing that $V$ is a redundant open set for the minimal system $(M,\varphi_p)$, meaning that $\varphi_p(M\setminus V)\supseteq \varphi_p(V)$ (see \cite[Lemma~2.1]{KST}). Indeed, since $V$ is disjoint with $I$, we have $\varphi_p(M\setminus V)\supseteq\varphi_p(I\cap M)= U$. Moreover, since the restriction $\varphi_p \colon J\to J'$ is a homeomorphism, $U=\varphi_p(V)$ by definition of $V$. Thus, indeed, $\varphi_p(M\setminus V) \supseteq \varphi_p(V)$ and this contradiction shows that all the acting maps $\varphi_s$ of $\Phi$ are homeomorphisms of $X$.
	
	Now we finish the proof of the theorem. Fix a free interval $J$ of $X$. By compactness of $X$ and minimality of $\Phi$, there is a finite set $F\subseteq S$ with $X=\bigcup_{r\in F}\varphi_r(J)$ (see Lemma~\ref{L:dens.back.orb}) and so $X$ is a union of finitely many free intervals. Thus, $X$ is a compact Hausdorff second countable $1$-dimensional manifold, and hence it is homeomorphic to a disjoint union of finitely many circles.
\end{proof}

So, for instance the Warsaw circle does not admit a minimal action of an abelian semigroup, although this follows also from Corollary~\ref{C:fixed-point-property} and the fact that the Warsaw circle has the fixed point property.
For an example which does not follow from Corollary~\ref{C:fixed-point-property},
take any compact Hausdorff space with a free interval having a circle as its retract.
(Notice that such a space does not have the fixed point property since the fixed point property is preserved by passing to a retract.)

\begin{theoremB}
	\thmB
\end{theoremB}
\begin{proof}
	By Theorem~A, $X$ is a circle (so we may assume that $X=\mathbb S^1$) and all the acting maps of $\Phi$ are homeomorphisms. Moreover, by Lemma~\ref{L:facts H}, the action $\Phi$ is free. Since every orientation reversing circle homeomorphism has a fixed point, we infer that all the acting homeomorphisms of $\Phi$ are orientation preserving.
	Consequently, the group	$G_\Phi$ defined by~\eqref{Eq:Gphi} is a subgroup
	of the group $\Homeo_+(\mathbb S^1)$ formed by the orientation preserving circle homeomorphisms.
	Since the natural action of $G_\Phi$ on $\mathbb S^1$ is minimal,
	$G_\Phi$ is conjugate to a group of	rotations by \cite[Corollary~5.15]{Ghy}.
	Hence $\Phi$ is conjugate to an $S$-action on $\mathbb S^1$ by rotations.
\end{proof}

\begin{corollary}\label{C:main-circle}
The following statements hold:
\begin{enumerate}
\item every minimal action of an abelian semigroup $S$ on the circle $\mathbb S^1$ is conjugate to an $S$-action by rotations;
\item every abelian subgroup of $\Homeo(\mathbb S^1)$ with a minimal natural action on $\mathbb S^1$ is isomorphic to a dense subgroup of $\mathbb S^1$.
\end{enumerate}
\end{corollary}
\begin{proof}
Part (1) follows immediately from Theorem~B. To verify part (2), fix an abelian subgroup $G$ of $\Homeo(\mathbb S^1)$ with a minimal natural action on $\mathbb S^1$. By part (1), $G$ is conjugate in $\Homeo(\mathbb S^1)$ to a group of rotations $G'$. Clearly, the natural action of $G'$ on $\mathbb S^1$ is also minimal. Let $H$ consist of all $\alpha\in\mathbb S^1$ such that the rotation of $\mathbb S^1$ by $\alpha$ is an element of $G'$. Then $H$ is a subgroup of $\mathbb S^1$ isomorphic to $G'$ and hence also to $G$. Finally, since the action of $G'$ on $\mathbb S^1$ is minimal, $H$ must be dense in $\mathbb S^1$.
\end{proof}


\section{Existence of minimal actions of a given abelian semigroup}

Our aim in this section is to give a necessary and sufficient condition for a given abelian semigroup $S$ to act in a minimal way on a disjoint union $X$ of finitely many circles. Our first step towards this goal is to reduce the original problem to the description of abelian subgroups $G$ of $\Homeo(X)$ whose natural action on $X$ is minimal. This is done in the following proposition.

\begin{proposition}\label{P:ex.min.gp}
Let $X$ be a disjoint union of finitely many circles. Given an abelian semigroup $S$, the following conditions are equivalent:
\begin{enumerate}
\item the semigroup $S$ acts in a minimal way on the space $X$;
\item there is a morphism of semigroups $h\colon S\to \Homeo(X)$ such that
the (abelian) subgroup $G$ of $\Homeo(X)$ generated by the image $h(S)$ of $h$ has a minimal natural action on $X$.
\end{enumerate}
\end{proposition}
\begin{proof}
First notice that, in an arbitrary group, the subgroup generated by an abelian subsemigroup is automatically abelian. In particular, the commutativity assumption on $G$ in (2) is superfluous.
	
We show that (2) follows from (1). To this end, let $\Phi$ be a minimal action of $S$ on $X$. By virtue of Theorem~A, all the acting maps $\varphi_s$ of $\Phi$ are homeomorphisms. Let $G=G_{\Phi}$ be the subgroup of $\Homeo(X)$ generated by $\varphi_s$ ($s\in S$). Then the map $h\colon S\ni s\mapsto\varphi_s\in G$ is a morphism of semigroups and its image $h(S)$ generates the group $G$ by definition of $G$. Finally, minimality of the natural action of $G$ on $X$ is immediate by minimality of $\Phi$ and so condition (2) holds.

We show that (1) follows from (2). So assume that $G$ is an (abelian) subgroup of $\Homeo(X)$ with a minimal natural action on $X$ and $h\colon S\to G$ is a morphism of semigroups whose image $h(S)$ generates $G$. Then $G=\{h(t)^{-1}h(s) : t,s\in S\}$ by commutativity of $G$. Now the semigroup $S$ acts on the space $X$ via homeomorphisms $\varphi_s=h(s)$ ($s\in S$). We show that this action $\Phi$ of $S$ on $X$ is minimal. To this end, fix a nonempty open set $U\subseteq X$. Since the group $G$ acts on $X$ in a minimal way and the space $X$ is compact, there exist $n\geq2$ and $t_i,s_i\in S$ ($i=1,\dots,n$) with $X=\bigcup_{i=1}^n(\varphi_{t_i}^{-1}\varphi_{s_i})(U)$ (see Lemma~\ref{L:dens.back.orb}).
Set $r=\sum_{j=1}^ns_j$ and $r_i=\sum_{j\neq i}s_j$ for $i=1,\dots,n$. Then
\begin{equation*}
X=\varphi_r^{-1}(X)=\bigcup_{i=1}^n\varphi_{t_i}^{-1}\varphi_r^{-1}(\varphi_{s_i}(U))=\bigcup_{i=1}^n(\varphi_{t_i+r_i})^{-1}(U)=\bigcup_{s\in S}\varphi_s^{-1}(U)
\end{equation*}
and the minimality of $\Phi$ thus follows.
\end{proof}

Now, in view of Proposition~\ref{P:ex.min.gp}, it remains to find an algebraic characterization of those abelian subgroups of $\Homeo(X)$, whose natural action on $X$ is minimal. Such characterization is given in the following theorem. Before formulating it, let us recall some facts from the theory of abelian groups (the reader is referred to \cite[Chapter~IX]{Fuchs} for details).

Let
\begin{equation*}
0\longrightarrow H\longrightarrow G\longrightarrow K\longrightarrow 0
\end{equation*}
be a short exact sequence of abelian groups. Then there is a symmetric $H$-valued \emph{$2$-cocycle} $f$ over $K$ such that $G$ is isomorphic to $K\times_f H$. This means that $f\colon K\times K\to H$ is a function satisfying the identities
\begin{itemize}
\item $f(k_1,k_2)=f(k_2,k_1)$,
\item $f(k_1,k_2)+f(k_1+k_2,k_3)=f(k_2,k_3)+f(k_2+k_3,k_1)$,
\item $f(0,k)=f(k,0)=0$,
\end{itemize}
and $K\times_f H$ is an abelian group, whose elements are pairs $(k,h)$ with $k\in K$, $h\in H$ and whose operation is given by the rule
\begin{equation*}
(k_1,h_1)+(k_2,h_2)=(k_1+k_2,h_1+h_2+f(k_1,k_2)).
\end{equation*}
We shall now use these facts in the proof of the following theorem.

\begin{theoremC}
	\thmC
\end{theoremC}
\begin{remark}
Condition (2) can be restated by saying that $G$ is an extension of a group with $\ell$ elements by a dense subgroup of $\mathbb S^1$. Let us also recall that a subgroup $H$ of $\mathbb S^1$ is dense if and only if it is infinite.
\end{remark}

\begin{proof}
First we verify implication (1)$\Rightarrow$(2). So assume that $G$ is a subgroup of $\Homeo(X)$ with a minimal natural action on $X$. Write $\mathbb S^1$ for a chosen component of $X$ and denote by $H$ the stabilizer of $\mathbb S^1$:
\begin{equation*}
H=\{\varphi\in G : \varphi(\mathbb S^1)=\mathbb S^1\}=\{\varphi\in G : \varphi(\mathbb S^1)\cap\mathbb S^1\neq\emptyset\}.
\end{equation*}
Clearly, $H$ is a subgroup of $G$ and the minimal action of $G$ on $X$ restricts to a minimal action of $H$ on $\mathbb S^1$. Moreover, since $G$ acts on $X$ freely by Lemma~\ref{L:facts H}, the restricted action of $H$ on $\mathbb S^1$ is effective and so $H$ can be identified with a subgroup of $\Homeo(\mathbb S^1)$ with a minimal natural action on $\mathbb S^1$. Thus, $H$ is isomorphic to a dense subgroup of $\mathbb S^1$ by Corollary~\ref{C:main-circle}. Let $K=G/H$ be the quotient group of $G$ by $H$. By definition of $H$, the elements of $K$ are in a one-to-one correspondence with the components of $X$ via the map $G/H\ni\varphi+H\mapsto\varphi(\mathbb S^1)\subseteq X$ for $\varphi\in G$. It follows that $K$ has cardinality $\ell$ and so the canonical short exact sequence
\begin{equation*}
0\longrightarrow H\longrightarrow G\longrightarrow G/H\longrightarrow 0
\end{equation*}
satisfies all the requirements from (2).

We verify implication (2)$\Rightarrow$(1). So let $G$ be an abelian group, $H$ be a dense subgroup of $\mathbb S^1$, $K$ be an abelian group with cardinality $\ell$ and assume that these groups fit into a short exact sequence
\begin{equation*}
0\longrightarrow H\longrightarrow G\longrightarrow K\longrightarrow 0.
\end{equation*}
Without loss of generality, we may assume that $X=K\times\mathbb S^1$ and $G=K\times_fH$, where $f\colon K\times K\to H$ is a symmetric $H$-valued $2-$cocycle over $K$. We define an action of $G$ on $X$ as follows: given $g\in G$, $g=(k,h)$, we define a map
\begin{equation*}
\varphi_g=\varphi_{(k,h)}\colon K\times\mathbb S^1\ni(a,b)\mapsto(k+a,b+h+f(k,a))\in K\times\mathbb S^1.
\end{equation*}
Clearly, all the maps $\varphi_g$ are continuous. Moreover, since $f$ is a symmetric $2$-cocycle over $K$, it follows that
\begin{itemize}
\item $\varphi_0$ is the identity on $K\times\mathbb S^1$,
\item $\varphi_{g'}\varphi_g=\varphi_{g'+g}$ for all $g,g'\in G$.
\end{itemize}
Thus, the maps $\varphi_g$ ($g\in G$) constitute an action of $G$ on $X$ by homeomorphisms. Since $\varphi_g$ possesses a fixed point only for $g=0$, the action of $G$ on $X$ is effective and hence $G$ may be identified with a subgroup of $\Homeo(X)$.

To finish the proof, we need to show that the action of $G$ on $X$ is minimal. This follows immediately from the following observations.
\begin{itemize}
\item The maps $\varphi_{(k,0)}$ ($k\in K$) permute the components of $X$ transitively. That is, given components $C,C'$ of $X$, there is $k\in K$ with $\varphi_{(k,0)}(C)=C'$.
\item The maps $\varphi_{(0,h)}$ ($h\in H$) form the stabilizer of $\mathbb S^1=0\times\mathbb S^1$ and they act on $\mathbb S^1$ via rotations by elements of $H$. Since $H$ is a dense subgroup of $\mathbb S^1$, the family of maps $\varphi_{(0,h)}$ ($h\in H$) acts on $\mathbb S^1$ in a minimal way.
\end{itemize}
\end{proof}

In Theorem~C we gave a necessary and sufficient condition for a given abelian group $G$ to act in a minimal and effective way on a disjoint union $X$ of $\ell$ circles. The following corollary of Theorem~C gives an analogous necessary and sufficient condition for the existence of a minimal (not necessarily effective) action of $G$ on $X$. It is based on a standard procedure of turning an action into an effective one; we present a detailed proof for completeness.

\begin{corollary}\label{C:min.act.not.eff}
Let $\ell$ be a positive integer and $X$ be a disjoint union of $\ell$ circles. Given an abelian group $G$, the following conditions are equivalent:
\begin{enumerate}
\item $G$ acts in a minimal way on $X$;
\item $G$ possesses a quotient group, which is an extension of a group with $\ell$ elements by a dense subgroup of $\mathbb S^1$.
\end{enumerate}
\end{corollary}
\begin{proof}
We begin by showing that (2) follows from (1). To this end, fix a minimal action $\Phi$ of $G$ on $X$. Denote by $h\colon G\to\Homeo(X)$ the morphism of groups induced by $\Phi$ and write $G'=\ker(h)$ for its kernel. Then $h$ factors through the canonical quotient morphism $p\colon G\to G/G'$ to a monomorphism of groups $h'\colon G/G'\to\Homeo(X)$, that is, $h=h'p$. The action $\Phi'$ of $G/G'$ on $X$ corresponding to $h'$ is then effective. Moreover, $\Phi$ and $\Phi'$ have the same set of orbits and, since $\Phi$ is minimal, so is $\Phi'$. Thus, the group $G/G'$ acts in a minimal and effective way on $X$ and hence, by virtue of Theorem~C, $G/G'$ is an extension of a group with $\ell$ elements by a dense subgroup of $\mathbb S^1$.

To see that (1) follows from (2), let $G'$ be a subgroup of $G$ such that the corresponding quotient group $G/G'$ is an extension of a group with $\ell$ elements by a dense subgroup of $\mathbb S^1$. Then, by virtue of Theorem~C, $G/G'$ acts on $X$ in a minimal and effective way. Fix such a minimal action $\Phi'$ of $G/G'$ on $X$ and denote by $h'\colon G/G'\to\Homeo(X)$ the morphism of groups induced by $\Phi'$. Consider the canonical quotient morphism $p\colon G\to G/G'$ and set $h=h'p$. Then $h\colon G\to\Homeo(X)$ is a morphism of groups. Denote by $\Phi$ the action of $G$ on $X$ corresponding to $h$. Then, similarly as above, the two actions $\Phi$ and $\Phi'$ have the same set of orbits and, since $\Phi'$ is minimal, so is $\Phi$. This verifies condition (1) and finishes the proof.
\end{proof}

Let us now illustrate how Theorem~C and Corollary~\ref{C:min.act.not.eff} can be used to detect the (non-)existence of minimal actions in concrete situations.

\begin{example}
The group $G=\mathbb Z$ acts in a minimal way on a disjoint union of $\ell$ circles for every $\ell\in\mathbb N$. Indeed, given $\ell\in\mathbb N$, the subgroup $H=\ell\mathbb Z$ of $\mathbb Z$ is isomorphic to a dense subgroup of $\mathbb S^1$ and the corresponding quotient group $G/H=\mathbb Z_\ell$ has cardinality $\ell$.
\end{example}

\begin{example}
We claim that the torsion subgroup $G=\tor(\mathbb S^1)$ of $\mathbb S^1$, consisting of the elements of $\mathbb S^1$ with a finite order, acts in a minimal way on a disjoint union of $\ell$ circles if and only if $\ell=1$. The ``if'' part is clear, since $\tor(\mathbb S^1)$ is dense in $\mathbb S^1$. To verify the ``only if'' part, fix $\ell\in\mathbb N$ and assume that $G$ acts in a minimal way on a disjoint union of $\ell$ circles. By virtue of Corollary~\ref{C:min.act.not.eff}, $G$ factors onto a group $K$ with cardinality $\ell$. Since $G$ is divisible, it follows that so is $K$. However, the only finite divisible abelian group is the trivial one, hence $\ell=1$.
\end{example}

\begin{example}
Let $(p_n)_{n\in\mathbb N}$ be an increasing sequence of prime numbers and $G$ be the direct sum $G=\bigoplus_{n\in\mathbb N}\mathbb Z_{p_n}$. We claim that $G$ acts in a minimal way on a disjoint union $X$ of $\ell$ circles if and only if $\ell$ is expressible as a product $\ell=\prod_{n\in F}p_n$ for a finite subset $F$ of $\mathbb N$ (the empty product being interpreted as $1$).

To verify the ``if'' part, let $F\subseteq\mathbb N$ be finite. Then $G$ splits into a direct sum $G=K\oplus H$, where $K=\bigoplus_{n\in F}\mathbb Z_{p_n}$ and $H=\bigoplus_{n\in\mathbb N\setminus F}\mathbb Z_{p_n}$. The group $H$ is infinite and is clearly isomorphic to a subgroup of $\tor(\mathbb S^1)$, hence also of $\mathbb S^1$. Moreover, the cardinality of $K$ is $\prod_{n\in F}p_n=\ell$. Finally, since $K$ is isomorphic to $G/H$, the group $G$ acts in a minimal (and effective) way on $X$ by Theorem~C.

To verify the ``only if'' part, fix $\ell\in\mathbb N$ and let $G$ act in a minimal way on $X$. Then Corollary~\ref{C:min.act.not.eff} yields a quotient group $K$ of $G$ with cardinality $\ell$. Let $q\colon G\to K$ be the underlying quotient morphism. Given $n\in\mathbb N$, we have either $\mathbb Z_{p_n}\subseteq\ker(q)$ or $\mathbb Z_{p_n}\cap\ker(q)=0$. Set $F=\{n\in\mathbb N : \mathbb Z_{p_n}\cap\ker(q)=0\}$. Since $q$ is monic on $\mathbb Z_{p_n}$ for every $n\in F$ and the prime numbers $p_n$ ($n\in F$) are mutually distinct, it follows that $q$ is monic also on $\bigoplus_{n\in F}\mathbb Z_{p_n}$. Also, $q$ vanishes on $\bigoplus_{n\in\mathbb N\setminus F}\mathbb Z_{p_n}$. Consequently, $q\colon \bigoplus_{n\in F}\mathbb Z_{p_n}\to K$ is an isomorphism and the cardinality of $K$ thus equals $\ell=\prod_{n\in F}p_n$.
\end{example}

\section{Minimal sets on compact spaces with a free interval}\label{S:min.set.free.int}

Our aim in this section is to prove the following trichotomy for minimal sets of abelian semigroup actions, which intersect a free interval.

\begin{theoremD}
	\thmD
\end{theoremD}
\begin{proof}
	Fix an arc $B\subseteq J$, whose interior in $X$ intersects $M$. By Lemma~\ref{L:dens.back.orb} applied to the restricted action of $S$ on $M$, there is a finite set $F\subseteq S$ with $M\subseteq\bigcup_{r\in F}\varphi_r(B)$. Being a finite union of Peano continua, $Y=\bigcup_{r\in F}\varphi_r(B)$ is a desired closed metrizable locally connected subspace of $X$ containing $M$.

	Assume that $M$ has a nonempty interior in $X$. As $M\subseteq Y$ and $Y$ is locally connected, $M$ contains a nonempty connected open subset $V$ of $X$. Further, since $M$ is a minimal set for $\Phi$ and it intersects $J$, there is $t\in S$ with $\varphi_t(V)\cap J\neq\emptyset$. By connectedness of $V$, there are two possibilities.
	\begin{enumerate}
		\item[(a)] The set $\varphi_t(V)\cap J$ contains an arc. Then the set $M\supseteq\varphi_t(V)\cap J$ contains a free interval and so it is a disjoint union of finitely many circles by Theorem~A. Thus, in this case, condition (2) holds.
		\item[(b)] The set $\varphi_t(V)$ is a singleton. Then the set $M$ is finite by virtue of Lemma~\ref{L:min.set.free.int}. Thus, in this case, condition (1) holds.
	\end{enumerate}

	Now assume that the set $M$ is infinite and nowhere dense in $X$ and write $J=(0,1)$. By the first step of the proof, $M$ is contained in the closed metrizable locally connected subspace $Y$ of $X$.
	Since the set $M$ has no isolated point by Lemma~\ref{L:min.set.free.int},
	$\conv(M\cap J)$ is a non-degenerate interval with end points $0\le a<b\le 1$; put $L=(a,b)$.
	Denote by $M^+$ (respectively, $M^-$) the set of all $x\in X\setminus L$ such that there is an increasing (respectively, decreasing) sequence $(x_n)_{n\in\mathbb N}$ in $M\cap L$ with $x_n\to x$.
	Since $M$ is compact metrizable, both $M^+$ and $M^-$ are non-empty.
	We show, in fact, that each of them is a singleton. We verify this for $M^+$, the argument for $M^-$ being similar. So let $x,y\in M^+$. Fix an increasing sequence $(x_n)_{n\in\mathbb N}$ in $M\cap L$ with $x_n\to x$. Given a connected neighbourhood $W$ of $x$ in $Y$, we have $x_n\in W$ for all but finitely many $n$ and so $W\supseteq(b-\varepsilon,b)$ for some $\varepsilon>0$ by connectedness of $W$. Since $Y$ is locally connected, each neighbourhood of $x$ in $Y$ contains a subset of $L$ of the form $(b-\varepsilon,b)$ with $\varepsilon>0$. However, the same argument applies to $y$ and so $x,y$ can not be separated by disjoint open sets in $Y$. In view of the Hausdorff property of $Y$, this means that $x=y$.

Write $A=\overline{\conv(M\cap J)}$; obviously, $A=\overline{L}$.
We claim that
\begin{equation*}
A=L\cup M^+\cup M^-.
\end{equation*}
The inclusion ``$\supseteq$'' is clear. To verify the converse inclusion, fix $x\in A\setminus L$. Then each neighbourhood of $x$ in $X$ intersects
$L\setminus[a+\varepsilon,b-\varepsilon]\subseteq(a,a+\varepsilon)\cup(b-\varepsilon,b)$
for every sufficiently small $\varepsilon>0$. Moreover, as observed above, each neighbourhood of $M^+\cup M^-$ in $X$ contains $(a,a+\varepsilon)\cup(b-\varepsilon,b)$ for some $\varepsilon>0$. Due to the finiteness of $M^+\cup M^-$ and the Hausdorff property of $X$, this means that $x\in M^+\cup M^-$.

Thus, $A$ is a (Hausdorff) compactification of an interval by at most two points, whence it follows that $A$ is either an arc or a circle. Further, since
$L$ is a free interval in $X$ and
$M$ is nowhere dense in $X$ by the assumption, the set $M\cap L$ is totally disconnected and hence so is $M\cap A=(M\cap L)\cup M^+\cup M^-$. Moreover, the compact metrizable set $M\cap A$ has no isolated points by Lemma~\ref{L:min.set.free.int}. Thus, $M\cap A$ is a Cantor set.
	
	Finally, since all the conditions (1)--(3) are clearly mutually exclusive, the proof of the theorem is finished.
\end{proof}

In the special situation when $X$ is a compact metric space and $S=\mathbb N$, we know from \cite[Theorem B]{DSS} that in case (3) the set $M$ is a cantoroid (that is, a compact metric space without isolated points, whose degenerate components form a dense set). For a general abelian semigroup $S$ we do not know whether this is also the case. However, if all the acting maps are homeomorphisms, even the following stronger result is true.

\begin{corollary}\label{C:min.set.free.int}
	Let $X$ be a compact Hausdorff space with a free interval $J$ and let $\Phi$ be an action of an abelian semigroup $S$ on $X$. Assume that $M\subseteq X$ is a minimal set for $\Phi$, which intersects~$J$. If all the acting maps $\varphi_s$ of $\Phi$ are homeomorphisms of $X$ then exactly one of the following conditions holds:
	\begin{enumerate}
		\item $M$ is finite;
		\item $M$ is open in $X$ and it is a disjoint union of finitely many circles;
		\item $M$ is a nowhere dense Cantor set in $X$.
	\end{enumerate}
\end{corollary}
\begin{proof}
	First we show that, in case (3) of Theorem~D, $M$ is a Cantor set. To this end, fix a Cantor set $C\subseteq M\cap J$ open in $M$. By Lemma~\ref{L:dens.back.orb} applied to the restricted action of $S$ on $M$, there is a finite set $F\subseteq S$ with $M=\bigcup_{r\in F}\varphi_r(C)$. Since all $\varphi_r$ are homeomorphisms of $X$, it follows that $M$ is a union of finitely many Cantor sets and hence $M$ itself is a Cantor set.
	
	Second we show that, in case (2) of Theorem~D, $M$ is open in $X$. To this end, notice that $J\subseteq M$. By Lemma~\ref{L:dens.back.orb}, there is a finite set $F\subseteq S$ with $M=\bigcup_{r\in F}\varphi_r(J)$. Since all $\varphi_r$ are homeomorphisms of $X$, it follows that $M$ is a union of open subsets of $X$ and hence $M$ itself is open in $X$. This finishes the proof.
\end{proof}

\end{large}
\end{document}